\documentclass[10pt]{article}
\font\smallit=cmti10

\usepackage{amssymb,latexsym,amsmath,epsfig,amsthm} 

\usepackage{geometry}
\usepackage{amsfonts}
\usepackage{colonequals}
\usepackage{stmaryrd}
\usepackage{tabularx}
\usepackage{xcolor}
\usepackage{float}
\usepackage{graphicx}
\usepackage[percent]{overpic}

\makeatletter

\renewcommand\section{\@startsection {section}{1}{\z@}
{-30pt \@plus -1ex \@minus -.2ex}
{2.3ex \@plus.2ex}
{\normalfont\normalsize\bfseries\boldmath}}

\renewcommand\subsection{\@startsection{subsection}{2}{\z@}
{-3.25ex\@plus -1ex \@minus -.2ex}
{1.5ex \@plus .2ex}
{\normalfont\normalsize\bfseries\boldmath}}

\renewcommand{\@seccntformat}[1]{\csname the#1\endcsname. }

\makeatother

\newtheorem{theorem}{Theorem}
\newtheorem{lemma}{Lemma}

\newtheorem{proposition}{Proposition}
\newtheorem{corollary}{Corollary}

\theoremstyle{definition}

\newtheorem{definition}[theorem]{Definition}

\newcommand{\ignore}[1]{}

\newcommand{\set}[1]{\left\{#1\right\}}
\newcommand{\setl}[1]{\left\{#1\right.\mathclose{}\kern-\nulldelimiterspace}
\newcommand{\setr}[1]{\kern-\nulldelimiterspace\mathopen{}\left.#1\right\}}

\newcommand{\fa}{\forall\:}
\newcommand{\ce}{\:\colonequals\:}

\newcommand{\abs}[1]{\left\lvert#1\right\rvert}
\newcommand{\norm}[1]{\left\Vert#1\right\Vert}
\newcommand{\floor}[1]{\left\lfloor#1\right\rfloor}

\newcommand{\fracpart}[1]{\left\{#1\right\}}

\newcommand{\N}{\mathbb{N}}
\newcommand{\Z}{\mathbb{Z}}

\newcommand{\R}{\mathbb{R}}
\newcommand{\C}{\mathbb{C}}
\newcommand{\Nz}{{\mathbb{N}}}

\renewcommand{\vector}[1]{{\mathbf{#1}}}

\newcommand{\D}[1]{{\mathcal{D}_{#1}}}
\newcommand{\Dstar}[1]{{\mathcal{D}_{#1}^{(\ast)}}}

\begin{document}

\begin{center}
\uppercase{\bf The finiteness property for shift radix systems with general parameters}
\vskip 20pt
{\bf Attila~Peth\H{o}\footnote{The author is supported in part by the OTKA grants NK104208, NK115479 and the FWF grant P27050. He also thanks the hospitality of the Montanuniversit\"at Leoben and the Technische Universit\"at Graz.}}\\
{\smallit Faculty of Informatics, University of Debrecen, Kassai \'{u}t 26, 4028 Debrecen, HUNGARY and University of Ostrava, Faculty of Science, Dvo\v{r}\'{a}kova 7, 70103 Ostrava, CZECH REPUBLIC}\\
{\tt petho.attila@inf.unideb.hu}\\
\vskip 10pt
{\bf J\"{o}rg~Thuswaldner\footnote{The author is supported by grant P29910 ``Dynamics, Geometry, and Arithmetic of Numeration'' of the Austrian Science Fund (FWF).}}\\
{\smallit Chair of Mathematics and Statistics,
University of Leoben, Franz Josef-Stra\ss{}e 18, 8700 Leoben, AUSTRIA}\\
{\tt joerg.thuswaldner@unileoben.ac.at}\\
\vskip 10pt
{\bf Mario~Weitzer\footnote{The author is supported by the grant P30205 ``Arithmetic Dynamical Systems, Polynomials, and Polytopes'' of the Austrian Science Fund (FWF).}}\\
{\smallit Institute of Analysis and Number Theory,
Graz University of Technology, Steyrergasse 30, 8010 Graz, AUSTRIA}\\
{\tt weitzer@math.tugraz.at}\\
\end{center}
\vskip 20pt

\centerline{%
} 
\vskip 30pt

\centerline{\bf Abstract}

\noindent
There are two-dimensional expanding shift radix systems (SRS) which have some periodic orbits. The aim of the present paper is to describe such unusual points as well as possible. We give all regions that contain parameters the corresponding SRS of which generate obvious cycles like $(1), (-1), (1,-1), (1,0), (-1,0)$. We prove that if $\mathbf{r}=(r_0,r_1)\in \R^2$ neither belongs to the aforementioned regions nor to the finite region $1\le r_0\le 4/3, -r_0 \le r_1 <r_0-1$, then $\tau_{\mathbf{r}}$ only has the trivial bounded orbit $\mathbf{0}$, which is a natural generalization of the established finiteness property for SRS with non-periodic orbits. The further reduction should be quite involving, because for all $1\le r_0< 4/3$ there exists at least one interval $I$ such that for the point $(r_0,r_1)$ this is not true whenever $r_1\in I$.

\pagestyle{myheadings}
\markright{%
}
\thispagestyle{empty}
\baselineskip=12.875pt
\vskip 30pt 

\section{Introduction}

The aim of this paper is to study properties of orbits of so-called \emph{shift radix systems}. These objects were introduced in 2005 by Akiyama et al.~\cite{Akiyama-Borbely-Brunotte-Pethoe-Thuswaldner:05}. We start by recalling their exact definition (for $x\in\mathbb{R}$ we use the notation $\lfloor x\rfloor$ and $\{x\}$ for its integer and fractional part, respectively).

\begin{definition}
For $d\in\N$ and $\vector{r}\in\R^d$ we call the mapping
\begin{align*}
\tau_\vector{r}:\Z^d&\to\Z^d,\\
\vector{a}=(a_0,\ldots,a_{d-1})&\mapsto(a_1,\ldots,a_{d-1},-\floor{\vector{r}\vector{a}})
\end{align*}
the \emph{$d$-dimensional shift radix system (SRS) associated with $\vector{r}$}.
\end{definition}

It is easy to see from this definition that  $\tau_\mathbf{r}$ is almost linear in the sense that it can be written as
\[
\tau_\mathbf{r}(\vector{a}) = R(\mathbf{r})\vector{a} + (0,\ldots,0,\{\mathbf{r}\mathbf{a}\})^t, \quad\hbox{where}\quad R(\mathbf{r})=\begin{pmatrix} \mathbf{0} & I_{d-1} \\ -r_0 & -r_1\; \cdots \; -r_d\end{pmatrix}
\]
with $I_{d-1}$ being the $(d-1)\times (d-1)$ identity matrix. However, the small deviation from linearity entails a rich dynamical behavior of $\tau_\mathbf{r}$ that has already been studied extensively in the literature. For a survey of different aspects of shift radix systems, we refer to \cite{KT:14}.

For $\vector{a}\in\Z^d$ the orbit of $\vector{a}$ under $\tau_\vector{r}$ is given by the sequence $(\tau_\vector{r}^n(\vector{a}))$, where $\tau_\vector{r}^n(\vector{a})$ stands for the $n$-fold application of $\tau_{\vector{r}}$ to $\vector{a}$. Given the definition of $\tau_\vector{r}$, the last $d-1$ entries of $\tau_\vector{r}^n(\vector{a})$ and the first $d-1$ entries of $\tau_\vector{r}^{n+1}(\vector{a})$ coincide for all $n\in\Nz$. Hence, we may choose to drop the redundant information and identify the orbit of $\vector{a}$ with the sequence of integers consisting of the entries of $\vector{a}$ followed only by the last entries of $\tau_\vector{r}^n(\vector{a})$. In other words, if $\vector{a}=(a_0,\ldots,a_{d-1})$ and $a_{d-1+n}$ is the last entry of $\tau_\vector{r}^n(\vector{a})$ for $n\in\N$, we identify the orbit of $\vector{a}$ with the sequence $(a_n)_{n\in\Nz}$. If $(a_n)$ ultimately consists only of zeroes, we call it a \emph{trivial} orbit of $\tau_\mathbf{r}$; otherwise, an orbit of $\tau_\mathbf{r}$ is called \emph{nontrivial}. An orbit $(a_n)$ is \emph{ultimately periodic} if there exist $n_0,p\in\N$, $p>0$, such that $a_{n+p}=a_{n}$ for $n \ge n_0$, and \emph{periodic} if this holds for $n_0=0$. In this case, we call $(a_n,\ldots, a_{n+p-1})$ with $n\ge n_0$ a \emph{cycle} of $\tau_\mathbf{r}$. The cycle $(0)$ is called \emph{trivial}, all other cycles are \emph{nontrivial}. The integer $p$ is called the \emph{period} of the cycle.

The properties of the orbits of $\tau_\vector{r}$ were studied extensively in the literature; see e.g. \cite{Akiyamaetal.2006b,Kirschenhoferetal.2010,Weitzer2015a}. In particular we define the sets
\[
\begin{split}
\mathcal{D}_d &:= \{\mathbf{r} \in \mathbb{R}^d \;:\; \hbox{each orbit of }\tau_\mathbf{r} \hbox{ ends up in a cycle} \}, \\
\mathcal{D}_d^{(0)} &:= \{\mathbf{r} \in \mathbb{R}^d \;:\; \hbox{each orbit of }\tau_\mathbf{r} \hbox{ ends up in the trivial cycle} \}.
\end{split}
\]

Elements of $\mathcal{D}_d^{(0)}$ are said to have the \emph{finiteness property}. As it has already been observed in \cite{Akiyama-Borbely-Brunotte-Pethoe-Thuswaldner:05}, for each $d\in\mathbb{N}$ both of these sets are contained in the closure of the so-called Schur-Cohn region (see~\cite{Schur:18})
\[
\mathcal{E}_d := \{(r_0,\ldots,r_{d-1}) \in \mathbb{R}^d \;:\;  \hbox{each root }  y \hbox{ of }x^d+r_{d-1}x^{d-1}+\cdots+r_0 \hbox{ satisfies } |y|<1 \}.
\]
In other words, the regions $\mathcal{D}_d$ and $\mathcal{D}_d^{(0)}$ only concern parameters corresponding to \emph{contractive} polynomials. In all these cases, the linear part $R(\mathbf{r})$ of $\tau_\vector{r}$ is contractive. Interesting results were also proved in the indifferent case, i.e., when all roots of the characteristic polynomial of $R(\mathbf{r})$ are on the unit circle (see \cite{Akiyama-Brunotte-Pethoe-Steiner:06,Akiyama-Brunotte-Pethoe-Steiner:07,Akiyama-Pethoe-:12, Kirschenhoferetal.2010, Kirschenhoferetal.2008}). Indeed, this is the difficult part of the description of the sets $\mathcal{D}_d$.

In this paper we focus our attention to the case for which $\tau_\vector{r}$ is expanding. The question we want to answer is the following: {\it for which values of $\vector{r}$ is the only cycle of $\tau_\vector{r}$ the trivial one?} In other words, we are going to study the set
\[
\mathcal{D}_d^{(*)} := \{\mathbf{r} \in \mathbb{R}^d \;:\;  \hbox{each ultimately periodic orbit of }\tau_\mathbf{r} \hbox{ ends up in the trivial cycle} \}.
\]
Indeed, it is clear that $\mathcal{D}_d^{(0)}\subset \mathcal{D}_d^{(*)}$. However, the reverse inclusion is not true since every $\tau_\mathbf{r}$ has unbounded orbits if $\mathbf{r}$ lies outside the closure of $\mathcal{E}_d$.

\begin{figure}[H]
\centering
\begin{overpic}[width=0.647\textwidth]{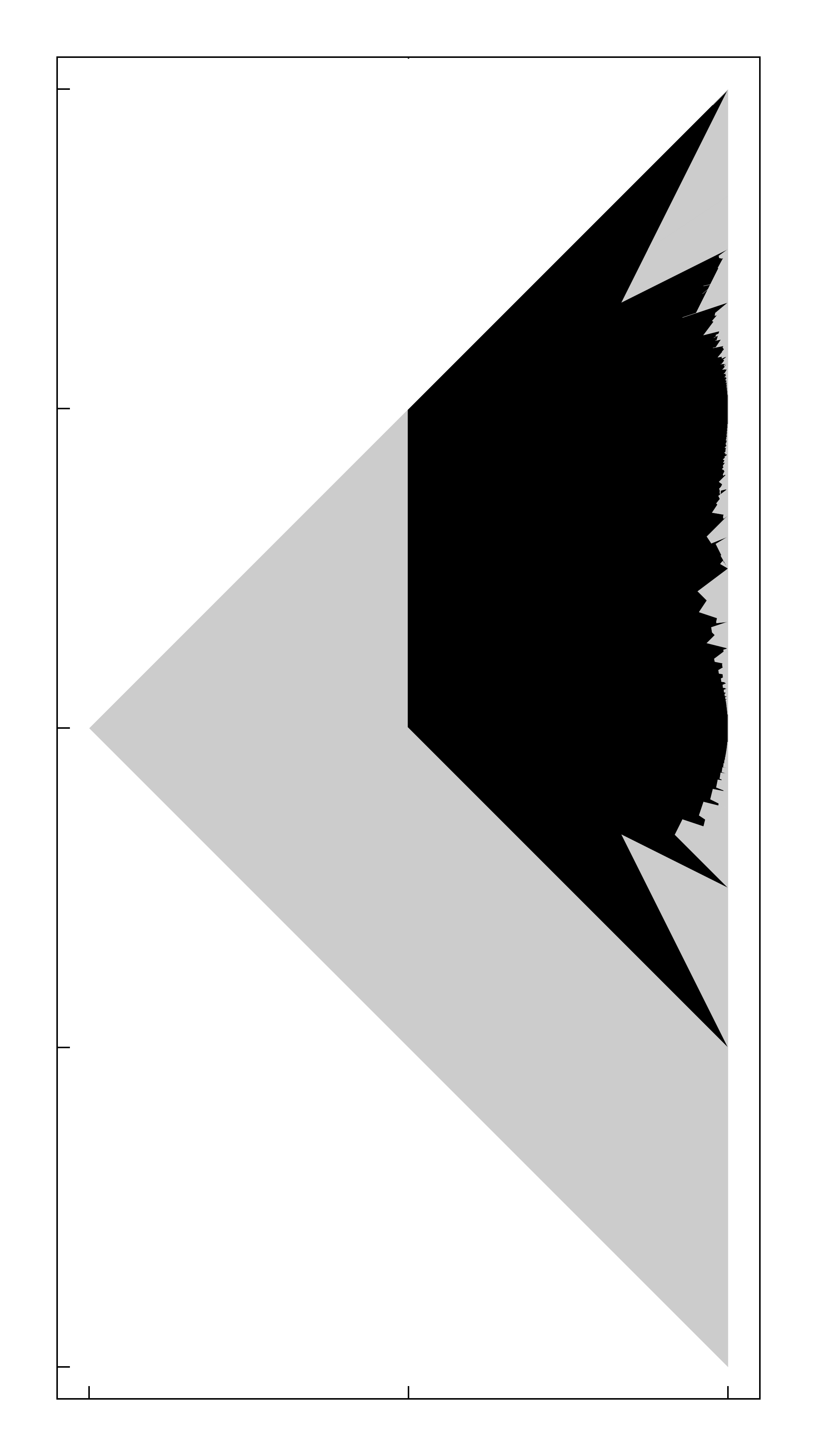}
\put(1.0,93.10){$2$}
\put(1.0,71.10){$1$}
\put(1.0,49.10){$0$}
\put(-1.05,27.10){$-1$}
\put(-1.05,5.10){$-2$}
\put(49.5,1.00){$1$}
\put(27.5,1.00){$0$}
\put(3.45,1.00){$-1$}
\end{overpic}
\caption{$\mathcal{D}_2^{(0)}$ (black) in $\mathcal{D}_2$ (gray).}
\label{fig:overviewD_2}
\end{figure}
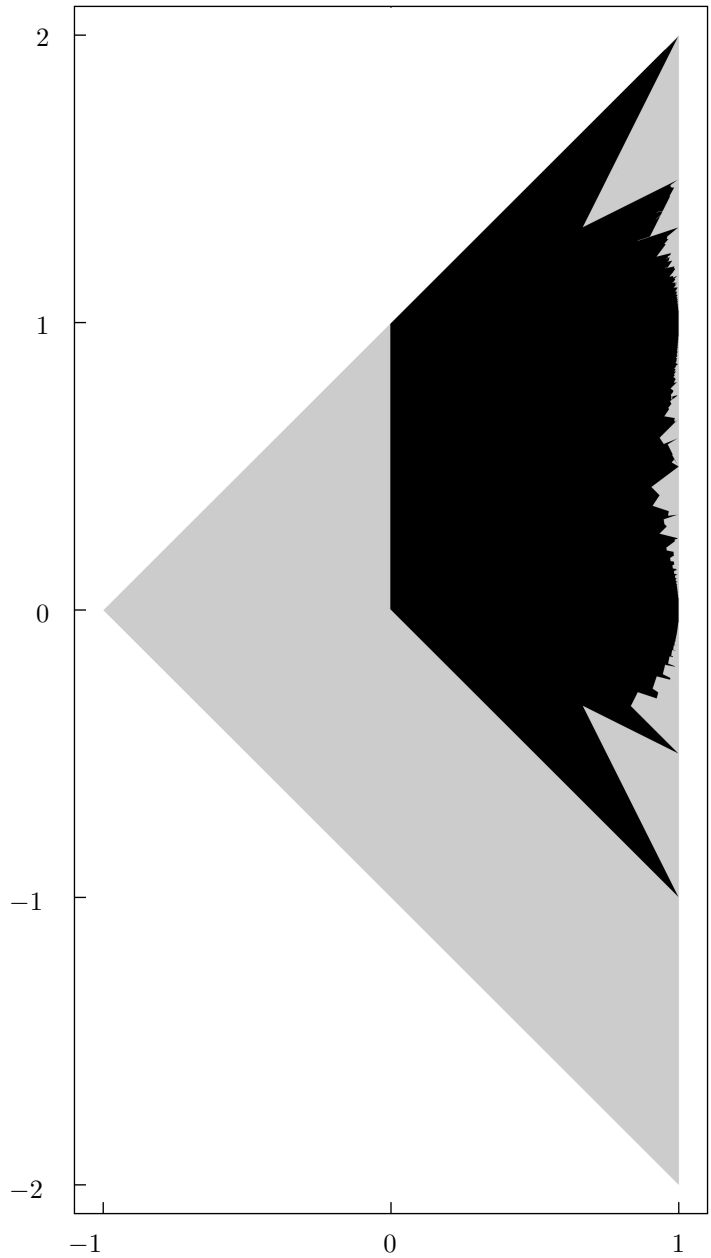

Our aim is to describe the sets $\mathcal{D}_d^{(*)} $. We are only dealing with the case $d=2$. Our investigations show that a complete description of $\mathcal{D}_d^{(*)} $ is already very hard (and even seems to be beyond reach) already for $d=2$.

If we define the sequence $(e_n)\in[0,1)^\Nz$ by $e_n\ce\fracpart{\vector{r}\tau_\vector{r}^n(\vector{a})}$, where $\fracpart{x}=x-\floor{x}$ denotes the fractional part of $x\in\R$, then it follows from the definitions that
\begin{align} \label{Eqnlrs}
a_{n+d}+r_{d-1}a_{n+d-1}+\cdots+r_0a_n=e_n\in[0,1)
\end{align}
for all $n\in\Nz$, i.e., $(a_n)$ is a nearly linear recursive sequence (from now on we will simply write \emph{nlrs} for a nearly linear recursive sequence).
In order to compute the sequence $(a_n)$, we repeatedly apply $\tau_\mathbf{r}$ to $(a_0,\dots,a_{d-1})$, where $\mathbf{r}=(r_0,\dots,r_{d-1})$ is a fixed vector of real numbers. We then study the properties of those sequences. Akiyama, Evertse and Peth\H{o} \cite{Evertse} considered nlrs from a different point of view. They called a sequence of complex numbers $(a_n)$ {\it nearly linear recursive} if there exist $p_0,\dots,p_{d-1}\in \C$ such that the "error sequence" $(e_n)$, defined by
$$
a_{n+d} +p_{d-1} a_{n+d-1} + \dots + p_0a_n = e_n,
$$
is bounded. Note that $(a_n)$ is a given in their setting. For a given nlrs $(a_n)$, the set of polynomials $B_tx +B_{t-1}x^{t-1}+\dots+B_0 \in \C[x]$, where the sequence $(\sum_{i=0}^{t} B_ia_{n+i})$ is bounded, is an ideal of the polynomial ring $\C[x]$, which is called the {\it ideal of $(a_n)$}. As $\C[x]$ is a principal ideal ring for all nlrs $(a_n)$, there exists a unique polynomial which generates the ideal of $(a_n)$. This is called the {\it characteristic polynomial} of $(a_n)$. The authors proved that all roots of the characteristic polynomial of an nlrs have an absolute value of at least one. They also proved that if one of the roots of the characteristic polynomial of $(a_n)$ lies outside the unit disc, then $(|a_n|)$ tends to grow exponentially.

If the mapping $\tau_\vector{r}$ is expanding, the polynomial $x^d+r_{d-1}x^{d-1}+\cdots+r_0$ has a root outside the unit disc. However, even if all of its roots lie outside the unit disc, it can happen that a bounded sequence of integers $(a_n)$ satisfies \eqref{Eqnlrs}. This happens for example in the case $d=2,p_1=-1.15, p_0=1.1$, when both roots of $x^2-1.15 x +1.1$ are larger than $1$, but the constant sequence $(1)$ satisfies \eqref{Eqnlrs}. It is easy to resolve this apparent contradiction: the characteristic polynomial of $(1)$ is, in the sense of \cite{Evertse}, not $x^2-1.15 x +1.1$, but the constant polynomial $1$.

In Section 2 we present preparatory results about bounded nlrs. SRSs are special cases of nlrs, but have particular features too. We collect them in Section 3, first in the general case, then specific to the case $d=2$. Section 4, which includes the characterization of $\mathcal{D}_2^{(*)}$, is divided into three subsections. First, we describe regions which do not belong to $\mathcal{D}_2^{(*)}$ because they have obvious cycles like $(1), (-1), (1,-1), (1,0), (-1,0)$. Next we prove that large regions belong to $\mathcal{D}_2^{(*)}$. Initially, we consider such regions for which the proof is simple. We can exclude the existence of cycles by proving that the orbits are monotonically increasing or decreasing. This always happens if the roots of $x^2+r_1x+r_0$ are real and at least one of them is positive. The hard cases are studied in Subsections 4.3 and 4.4. We are able to reduce the uncertain region to a bounded one by using estimates of the size of elements of a cycle, depending on the size of the roots of $x^2+r_1x+r_0$. Finally, we further reduce the uncertain region to $1\le r_0\le 4/3, -r_0 \le r_1 <r_0-1$ by using a Brunotte-type algorithm. The further reduction should be involving because, for all $1\le r_0< 4/3$, there exists at least one interval $I$ such that $(r_0,r_1)\notin \mathcal{D}_2^{(*)}$ whenever $r_1\in I$.

\begin{figure}[H]
\centering
\includegraphics[width=0.75\textwidth]{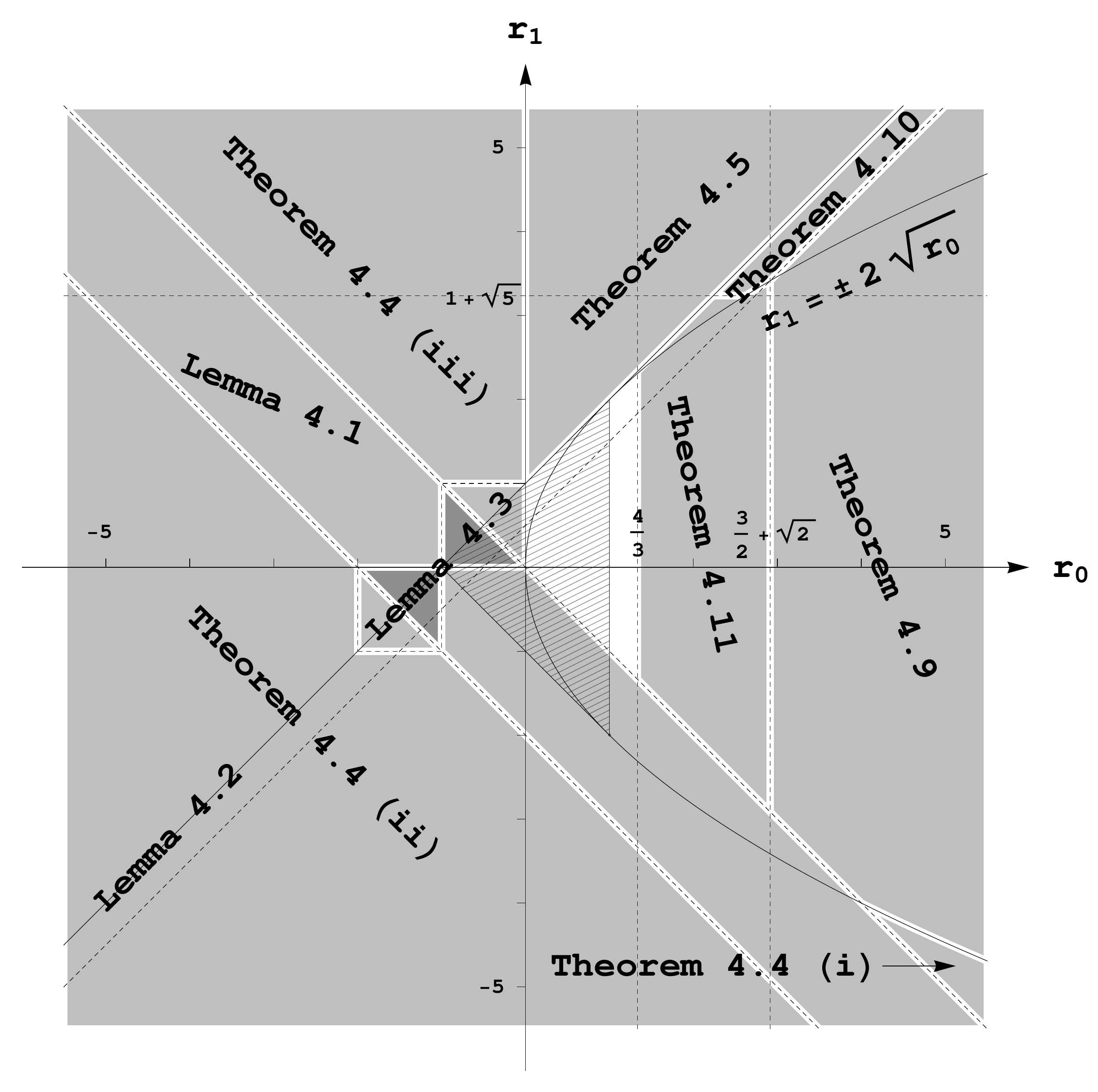}\\
\includegraphics[width=0.75\textwidth]{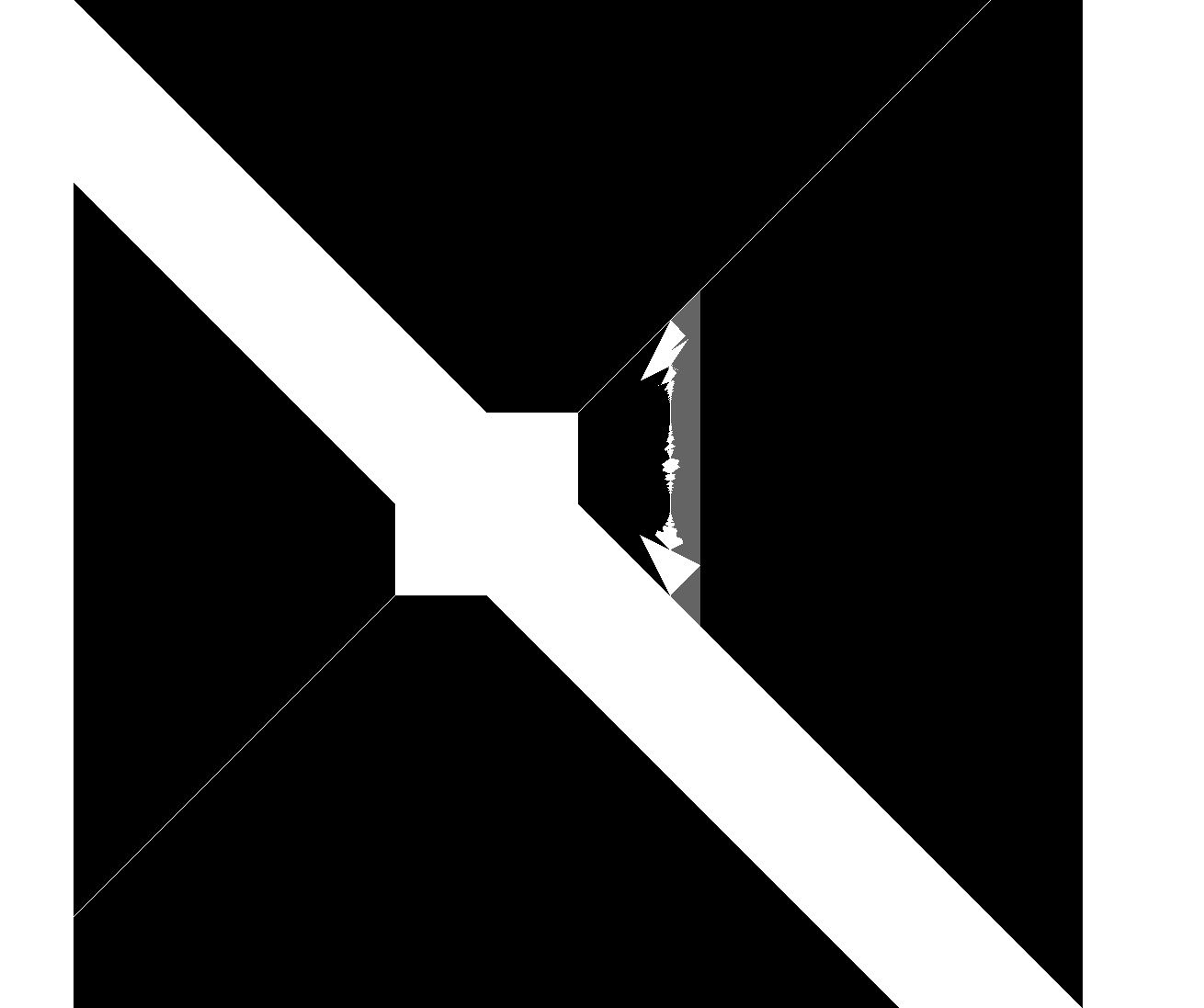}
\caption{The top image shows which regions are covered by the different theorems. The bottom image shows the resulting shape of $\Dstar{2}$. Black areas belong to $\Dstar{2}$, white areas do not and gray areas have not been settled by now.}
\label{fig:overview}
\end{figure}

\newpage

\section{General results on bounded nearly linear recursive sequences}

The present section contains some preparatory results that are stated in the more general framework of nlrs. These sequences were studied thoroughly in the recent paper~\cite{Evertse}. For the sake of completeness, we recall the definition of these objects. A sequence $(a_n)$ is called \emph{nearly linear recursive} if there exist $p_0,\ldots, p_{d-1} \in \C$ such that the sequence $(e_n)$, defined by
\[
a_{n+d} + p_{d-1}a_{n+d-1} + \cdots + p_0a_n = e_n,
\]
is bounded.

Our Theorem~\ref{TBoundedSeq} is a kind of complement to \cite[Theorem~1.1]{Evertse}: in the terminology of this paper it deals with nlrs with constant characteristic polynomials.

First, we introduce the following lemma:

\begin{lemma}
\label{LSeqBounded}
Let $\beta\in\C$ with $\abs{\beta}\neq1$ and $(b_n)\in\C^{\Nz}$ such that
\begin{align}
\label{ESeqA}
|b_{n+1}-\beta b_n|\le E
\end{align}
for all $n\in\Nz$. Then the following assertions hold:
\begin{itemize}
\item[(i)] If $\abs{\beta}<1$, then for each $\varepsilon>0$ there is $n_0\in\mathbb{N}$ such that $\abs{b_n} < \frac{E}{1-\abs{\beta}}+ \varepsilon$ holds for $n \ge n_0$. In particular, $(|b_n|)$ is bounded.
\item[(ii)] If $\abs{\beta}>1$ and $(|b_n)|$ is bounded, then $\abs{b_n}\leq\frac{E}{\abs{\beta}-1}$ for all $n \in \Nz$.
\end{itemize}
\end{lemma}

\begin{proof}
Set
$$
b_{\ell+1} - \beta b_{\ell} = e_{\ell}.
$$
Multiplying this by $\beta^j$ for $\ell=n+k-j-1$, $j=0,\ldots,n-1$ and summing up the resulting equations yields
\begin{align}
\label{ESeqB}
b_{n+k}-\beta^n b_k=e_{n+k-1}+\beta e_{n+k-2}+\cdots+e_k\beta^{n-1}
\end{align}
for all $n\in\Nz$ and $k\in\N$.

If $\abs{\beta}<1$ and $k=1$ we thus get
\begin{align}
\abs{b_{n+1}}&\leq\sum_{j=1}^n\abs{e_j}\abs{\beta}^{n-j}+\abs{\beta}^n\abs{b_1}<\frac{E}{1-\abs{\beta}}+\abs{\beta}^n\abs{b_1}.
\end{align}
Since  $\abs{\beta}^n\abs{b_1}$ tends to $0$ for $n\to\infty$,(i) is proven.

In order to prove (ii), let us assume that $\abs{\beta}>1$ and that there exists $B > 0$ satisfying $\abs{b_n}\leq B$ for all $n\in\Nz$. Upon division of both sides of \eqref{ESeqB} by $\beta^n$, we get
\begin{align*}
\abs{b_k}=\abs{-\frac{1}{\beta}\sum_{j=0}^{n-1}e_{k+j}\beta^{-j}+\frac{b_{n+k}}{\beta^n}}<\frac{E}{\abs{\beta}}\sum_{j=0}^\infty\frac{1}{\abs{\beta}^j}+\frac{\abs{b_{n+k}}}{\abs{\beta}^n}\leq\frac{E}{\abs{\beta}-1}+\frac{B}{\abs{\beta}^n}.
\end{align*}
This proves (ii) because $k$ is fixed, and ${B}/{\abs{\beta}^n}$ tends to $0$ for $n\to\infty$.
\end{proof}

The following result, which is of interest in its own right, contains bounds for certain nlrs.

\begin{theorem}
\label{TBoundedSeq}
Let $d\in\N$ and $\beta_1,\ldots,\beta_d\in\C$ such that $\abs{\beta_1}\leq\cdots\leq\abs{\beta_r}<1<\abs{\beta_{r+1}}\leq\cdots\leq\abs{\beta_d}$ for some $r\in\set{1,\ldots,d}$. Furthermore, let $(a_n),(e_n)\in\C^{\Nz}$ with $\abs{e_n}\leq E$ for all $n\in\Nz$ and some $E>0$, such that
\begin{align}
\label{ESeqC}
a_{n+d}+p_{d-1}a_{n+d-1}+\cdots+p_0a_n=e_n
\end{align}
for all $n\in\Nz$, where $(x-\beta_1)\cdots(x-\beta_d)=x^d+p_{d-1}x^{d-1}+\cdots+p_1x+p_0$. Then,
\begin{itemize}
\item[(i)] If $r=d$, then for each $\varepsilon>0$ there exists $n_0\in\mathbb{N}$ such that
$$
\abs{a_n}<\frac{E}{\prod_{j=1}^d(1-\abs{\beta_j})} +\varepsilon
$$
for $n \ge n_0$. In particular, $(|a_n|)$ is bounded.
\item[(ii)] If $r<d$ and $(|a_n|)$ is bounded, then for each $\varepsilon> 0$ there exists $n_0\in\mathbb{N}$ such that
$$
\abs{a_n}<\frac{E}{\prod_{j=1}^d\abs{1-\abs{\beta_j}}} + \varepsilon
$$
for
$n \ge n_0$.
\end{itemize}
\end{theorem}

\begin{proof}
The assertion is true for $d=1$ by Lemma~\ref{LSeqBounded}. Assume that it is also true for $d-1$ and let
\begin{align}
s(x)&\ce x-\beta_d
\end{align}
and $q(x)=x^{d-1}+q_{d-2}x^{d-2}+\ldots+q_1x+q_0\in\C[x]$ such that
\begin{align}
p(x)=q(x)s(x).
\end{align}
Furthermore, let
\begin{align}
b_n&\ce a_{n+1}-\beta_da_n
\end{align}
for all $n\in\N$ and let $\sigma$ denote the shift operator. Then,
\begin{align}
p\left(\sigma\right)(a_n)=q\left(\sigma\right)s\left(\sigma\right)(a_n)=q\left(\sigma\right)(b_n)
\end{align}
for all $n\in\N$ which, together with \eqref{ESeqC}, yields
\begin{align}
e_n&=\sum_{j=0}^dp_{d-j}a_{n+d-j}
=\sum_{j=1}^dq_{d-j}b_{n+d-j}
\end{align}
for all $n\in\N$.

\ignore{
\begin{align}
S_j(x_1,\ldots,x_k)\ce\sum_{1\leq i_1<\cdots<i_j\leq k}x_{i_1}\cdots x_{i_j}
\end{align}
for all $k\in\N$ and $j\in\set{1,\ldots,k}$, the $j$-th elementary symmetric polynomial on the indeterminates $x_1,\ldots, x_k$. Furthermore, we set $S_0(x_1,\ldots,x_k)\ce1$ for all $k\in\N$ and $S_j(x_1,\ldots,x_k)\ce0$ for all $k\in\N$ and for all $j\in\Z\setminus\set{0,\ldots,k}$. It is well known that
\begin{align}
p_{d-j}=(-1)^j S_j(\beta_1,\ldots,\beta_d)
\end{align}
for all $j\in\set{1,\ldots,d}$. Clearly,
\begin{align}
S_j(x_1,\ldots,x_k)&=x_k S_{j-1}(x_1,\ldots,x_{k-1})+S_j(x_1,\ldots,x_{k-1})
\end{align}
which, together with \eqref{ESeqC}, yields
\begin{align*}
e_n &= \sum_{j=0}^d p_{d-j} a_{n+d-j} \\
&= \sum_{j=0}^d (-1)^j S_j(\beta_1,\dots,\beta_d) a_{n+d-j} \\
&=\sum_{j=0}^d(-1)^j(\beta_dS_{j-1}(\beta_1,\ldots,\beta_{d-1})+S_j(\beta_1,\ldots,\beta_{d-1}))a_{n+d-j}\\
&=\sum_{j=0}^{d-1}(-1)^jS_j(\beta_1,\ldots,\beta_{d-1})(a_{n+d-j}-\beta_da_{n+d-j-1})\\
&=\sum_{j=0}^{d-1}(-1)^jS_j(\beta_1,\ldots,\beta_{d-1})b_{n+d-j-1},
\end{align*}
where $b_n\ce a_{n+1}-\beta_da_n$ for all $n\in\N$.
}

Clearly, if $(|a_n|)$ is bounded, then so is $(|b_n|)$ and the assumptions of the theorem hold for both sequences. By the induction hypothesis we get that for each $\varepsilon > 0$
\begin{align}
\abs{b_n}<\frac{E}{\prod_{j=1}^{d-1}\abs{1-\abs{\beta_j}}}+ \varepsilon
\end{align}
for all large enough $n\in\Nz$. Thus, by Lemma \ref{LSeqBounded},
\begin{align}
\abs{a_n}<\frac{E}{\prod_{j=1}^d\abs{1-\abs{\beta_j}}}+ \varepsilon
\end{align}
holds for all large enough $n\in\Nz$.
\end{proof}


\section{Bounded orbits of expansive shift radix systems}

\subsection{General SRS}
$ $\\
The situation for expanding and contractive $\tau_\vector{r}$ is related, as the following consequence of Theorem~\ref{TBoundedSeq} indicates.

\begin{corollary} \label{Cor23}
Assume that the sequence of integers $(a_n)$ satisfies \eqref{Eqnlrs} for all $n\in\Nz$. Let $x^d+r_{d-1}x^{d-1}+\cdots+r_1x+r_0=(x-\beta_1)\cdots(x-\beta_d)$ where $\beta_1,\ldots,\beta_d\in\C$ such that $\abs{\beta_1}\leq\cdots\leq\abs{\beta_r}<1<\abs{\beta_{r+1}}\leq\cdots\leq\abs{\beta_d}$ for some $r\in\set{0,\ldots,d}$. Then,
\begin{itemize}
\item[(i)] if $r=d$, or,
\item[(ii)] if $r<d$ and $(|a_n|)$ is bounded,
\end{itemize}
then $(a_n)$ is ultimately periodic and
$$
\abs{a_n}\le \frac{1}{\prod_{j=1}^d\abs{1-\abs{\beta_j}}}
$$
holds for all elements of the cycle.

\end{corollary}

\begin{proof}
If $r=d$, then Theorem~\ref{TBoundedSeq} (i) implies that $(a_n)$ is bounded. In the case $r<d$ the boundedness of $(a_n)$ is part of the assumptions. Since $(a_n)$ is a bounded sequence of integers, there exist $j<k$ such that $a_{j+i} = a_{k+i}$ holds for each $i\in \{0,\dots, d-1\}$. However, as a consequence of \eqref{Eqnlrs}, $a_{j+d}$ is uniquely defined in terms of $a_j,\ldots, a_{j+d-1}$ and $a_{k+d}$ is uniquely defined in terms of $a_k,\ldots, a_{k+d-1}$. Thus, $a_{j+d}=a_{k+d}$ and, by induction, $a_{j+n} = a_{k+n}$ for all $n \ge 0$. As a consequence, our sequence is ultimately periodic.

Let $a_j$  be an element of the cycle and choose $\varepsilon>0$ arbitrarily. Then, according to Theorem~\ref{TBoundedSeq}, there exists an $n_0 \in \N$ such that
\begin{equation}\label{eq:pereps}
\abs{a_n}< \frac{1}{\prod_{j=1}^d\abs{1-\abs{\beta_j}}} + \varepsilon
\end{equation}
for $n \ge n_0$. Since $a_j$ is an element of the cycle of $(a_n)$, there is an index $n\ge n_0$ such that $a_j=a_n$. Hence, the estimate in \eqref{eq:pereps} also holds for $a_j$. However, given that we choose $\varepsilon$ arbitrarily, we even get
$$
\abs{a_j}\le \frac{1}{\prod_{j=1}^d\abs{1-\abs{\beta_j}}}
$$
and the proof is finished.
\end{proof}

\subsection{Specialization to two-dimensional SRS}
$ $\\
For the remaining part of the section let $d=2$, $\vector{r}=(r_0,r_1)\in\R^2$, $\vector{a}=(a_0,a_1)\in\Z^2$, $(a_n)$ the orbit of $\vector{a}$ under $\tau_\vector{r}$, and let $(e_n)$ be the corresponding error sequence, i.e.,
\begin{equation}\label{Eqrek}
  a_{n+2}+r_1a_{n+1}+r_0a_n = e_n, \; e_n \in [0,1)
\end{equation}
for all $n\in\Nz$. Furthermore, define $\alpha_1,\alpha_2\in\C$ by
\begin{align}
x^2+r_1x+r_0=(x-\alpha_1)(x-\alpha_2).
\end{align}
Then, $r_0=\alpha_1\alpha_2$ and $r_1=-(\alpha_1+\alpha_2)$.

\begin{proposition} \label{Prop25}
Let $\abs{\alpha_1},\abs{\alpha_2}\neq1$ and assume that $(|a_n|)$ is bounded and, hence, ultimately periodic. Then all elements $a_n$, which do not belong to the preperiod, satisfy
\begin{align} \label{rek1}
\abs{a_{n+1}-\alpha_1a_n}\le\frac{1}{\abs{\abs{\alpha_{2}}-1}},\; \abs{a_{n+1}-\alpha_2a_n}\le\frac{1}{\abs{\abs{\alpha_{1}}-1}}
\end{align}
and
\begin{align} \label{rek2}
\abs{a_n}\le \frac{1}{\abs{\abs{\alpha_1}-1}\abs{\abs{\alpha_2}-1}}.
\end{align}
\end{proposition}

\begin{proof}
We have
\begin{equation}\label{eq:enbn}
\begin{split}
e_n&=a_{n+2}-(\alpha_1+\alpha_2)a_{n+1}+\alpha_1\alpha_2a_n\\
&=(a_{n+2}-\alpha_1a_{n+1})-\alpha_2(a_{n+1}-\alpha_1a_n)=b_{n+1}-\alpha_2 b_n,
\end{split}
\end{equation}
where $b_n\ce a_{n+1}-\alpha_1a_n, n\ge 0$. Following Corollary \ref{Cor23}, $(a_n)$ is ultimately periodic, hence, $(b_n)$ is ultimately periodic as well. With the sequences $(b_n)$ and $(e_n)$, and given the choices $\beta=\alpha_2$ and $E=1$, the assumptions of Lemma~\ref{LSeqBounded} are satisfied. Thus, for each $\varepsilon > 0$, we have
$$
|a_{n+1}-\alpha_1a_n| = |b_n| < \frac{1}{|1-|\alpha_2||} + \varepsilon
$$
for $n$ sufficiently large. For $b_n$ in the cycle we can get rid of the summand $\varepsilon$ following the same reasoning as in the proof of Corollary~\ref{Cor23}. Thus, we get \eqref{rek1} for $i=1$. By interchanging $\alpha_1$ and $\alpha_2$ we get the case $i=2$.

Inequality \eqref{rek2} is a special instance of Corollary \ref{Cor23}.
\end{proof}

The following result is an immediate consequence of Proposition~\ref{Prop25}.

\begin{corollary}
If $\frac{1}{\abs{\abs{\alpha_1}-1}}\frac{1}{\abs{\abs{\alpha_2}-1}}<1$, then each orbit of $\tau_\vector{r}$ is either unbounded or ends up in the cycle $(0)$.
\end{corollary}

Proposition~\ref{Prop25} holds for arbitrary complex numbers $\alpha_1,\alpha_2$. If they are real, then we can improve inequality \eqref{rek1} a bit, which will turn out useful in later applications.

\begin{proposition} \label{Prop26}
Let $\alpha_1,\alpha_2\neq \pm1$ be real numbers. Assume that $(|a_n|)$ is bounded and, hence, ultimately periodic. Then all elements $a_n$ contained in the cycle satisfy
\begin{align}
0 \le& a_{n+1}-\alpha_1a_n <\frac{1}{1-\alpha_{2}}, \; \mbox{if} \quad 0\le \alpha_2<1,\label{eqProp261}\\
-\frac{1}{\alpha_{2}-1} <& a_{n+1}-\alpha_1a_n\le 0 , \; \mbox{if} \quad \alpha_2>1,\label{eqProp262}\\
\frac{\alpha_2}{1-\alpha_2^2} <& a_{n+1}-\alpha_1a_n<\frac{1}{1-\alpha_{2}^2}, \; \mbox{if} \quad  -1< \alpha_2<0,\label{eqProp263}\\
\frac{-1}{\alpha_2^2-1} <& a_{n+1}-\alpha_1a_n<\frac{-\alpha_2}{\alpha_{2}^2-1}, \; \mbox{if} \quad \alpha_2<-1.\label{eqProp264}
\end{align}
\end{proposition}

\begin{proof}
  We may assume w.l.o.g. that $(a_n)$ is purely periodic with period length $p$. Set, as in the proof of Proposition \ref{Prop25}, $b_n\ce a_{n+1}-\alpha_1 a_n$ for $n\ge 0$. Then, $(b_n)$ is periodic as well with period length $p$. Consider the equations
  \begin{equation*}
    b_{j+1} -\alpha_2 b_j = e_j
  \end{equation*}
for $j = 0,\dots,p-1$ (c.f. \eqref{eq:enbn}) and note that $b_p=b_0$. Multiplying the equations by appropriate powers of $\alpha_2$ we get
\begin{equation}\label{eqProp265}
\alpha_2^j b_{p-j} - \alpha_2^{j+1} b_{p-j-1} = \alpha_2^j e_{p-j-1}, \; j=0,\dots,p-1
\end{equation}
which yields after summation
$$
b_0(1-\alpha_2^p) = b_p-\alpha_2^pb_0 = \sum_{j=0}^{p-1}\alpha_2^j(b_{p-j}-\alpha_2b_{p-j-1})=\sum_{j=0}^{p-1}\alpha_2^je_j.
$$
If $\alpha_2>0$, we get
$$
0\le b_0(1-\alpha_2^p) < \sum_{j=0}^{p-1} \alpha_2^j = \frac{1-\alpha_2^p}{1-\alpha_2},
$$
due to $0\le e_j<1$. Distinguishing the cases $0<\alpha_2<1$ and  $\alpha_2>1$, we get the first two inequalities.

\medskip

If $\alpha_2<0$, we assume for simplicity that $p$ is even, say $p=2p_1$. This is allowed because if $p$ is a period length of a sequence, then $2p$ is a period length too. Equation~\eqref{eqProp265} implies that
$$
0\le \alpha_2^j b_{p-j} - \alpha_2^{j+1} b_{p-j-1} < \alpha_2^j
$$
if $j$ is even, and
$$
\alpha_2^j < \alpha_2^j b_{p-j} - \alpha_2^{j+1} b_{p-j-1} \le 0
$$
if $j$ is odd. Summing the inequalities above for $j=0,\dots,2p_1-1$, we obtain
$$
\alpha_2 \sum_{j=0}^{p_1-1} \alpha_2^{2j} < b_0(1-\alpha_2^p) < \sum_{j=0}^{p_1-1} \alpha_2^{2j}.
$$
Using $$\sum_{j=0}^{p_1-1} \alpha_2^{2j} = \frac{1-\alpha_2^p}{1-\alpha_2^2}$$ and distinguishing the cases $-1<\alpha_2<0$ and  $\alpha_2<-1$, we get the last two inequalities.
\end{proof}

\section{Characterization of $\mathcal{D}_2^{(*)}$}

\subsection{Regions outside of $\mathcal{D}_2^{(*)}$}

\begin{lemma}
\label{LSameSign}
The mapping $\tau_\vector{r}$ has a nontrivial cycle whose elements all have the same sign if and only if $-2<r_0+r_1<0$. Thus, the set $\{(r_0,r_1)\in \mathbb{R}^2\,:\, -2<r_0+r_1<0\}$ has an empty intersection with $\mathcal{D}_2^{(*)}$

In particular, if $-1 \le r_0+r_1 < 0$, then $(1)$ is a cycle of $\tau_\vector{r}$, and if $-2 < r_0+r_1 < -1$, then $(-1)$ is a cycle of $\tau_\vector{r}$.
\end{lemma}

\begin{proof}
Assume that $\tau_\vector{r}$ has a nontrivial cycle $(a_0,\ldots,a_{p-1})$ whose members have the same sign. Then,
\begin{align}
0\leq a_ir_0+a_{i+1}r_1+a_{i+2}<1
\end{align}
for all $i\in\set{0,\ldots,p-1}$. Summing up these inequalities and taking into account that $a_p=a_0$ and $a_{p+1}=a_1$, we get
\begin{align}
0\leq\sum_{i=0}^{p-1}a_i(r_0+r_1+1)<p.
\end{align}
Since all $a_i$ have the same sign (and cannot be $0$ by nontriviality of the cycle), it follows that
\begin{align}
\abs{\sum_{i=0}^{p-1}a_i}\geq p
\end{align}
and, hence,
\begin{align}
-1<r_0+r_1+1<1.
\end{align}
If $-1\leq r_0+r_1<0$, i.e., $0\leq r_0+r_1+1<\frac{1}{h}$ for $0\neq h\in\N$, it follows that $(t)$, where $0<t\leq h$, is a cycle of $\tau_\vector{r}$. In particular, $(1)$ is a cycle of $\tau_\vector{r}$ if $-1\leq r_0+r_1<0$.

If $-2<r_0+r_1<-1$, i.e., $-\frac{1}{h}<r_0+r_1+1<0$ for $0\neq h\in\N$, then $0<tr_0+tr_1+t<-\frac{t}{h}\leq1$ for $-h\leq t<0$, hence, $(t)$ is a cycle of $\tau_\vector{r}$. In particular, $(-1)$ is a cycle of $\tau_\vector{r}$ if $-1\leq r_0+r_1+1<0$.
\end{proof}

\begin{lemma}
\label{LAltSign}
If $\tau_\vector{r}$ has a cycle of alternating signs, then $\abs{r_0-r_1+1}<\frac{1}{2}$. Furthermore, $\tau_\vector{r}$ has a cycle of the form $(t,-t)$ for some $t\in\Z\setminus\set{0}$ if and only if $r_0-r_1+1=0$. Thus, $\{(r_0,r_0+1)  \mid r_0\in \mathbb{R} \}$ has empty intersection with $\mathcal{D}_2^{(*)}$.
\end{lemma}

\begin{proof}
Assume that $\tau_\vector{r}$ has a cycle $(a_0,\ldots,a_{p-1})$ with $a_ia_{i+1}<0$ for all $i\in\set{0,\ldots,p-1}$. Then, the period $p$ has to be even and we may assume w.l.o.g. that $a_0>0$. By definition,
\begin{align} \label{e4_5}
0\leq a_ir_0+a_{i+1}r_1+a_{i+2}<1
\end{align}
for all $i\in\set{0,\ldots,p-1}$. Multiplying the inequalities corresponding to odd values of $i$ by $-1$ and summing over the whole cycle, we get
\begin{align}
-\frac{p}{2}<\sum_{i=0}^{p-1}(-1)^ia_i(r_0-r_1+1)<\frac{p}{2}.
\end{align}
Observing that $\sum_{i=0}^{p-1}(-1)^ia_i\ge p$ for all cycles $(a_0,\ldots,a_{p-1})$ with members of alternating signs, we obtain the first statement.

Assume that $\tau_\vector{r}$ admits a cycle $(t,-t)$ with $t\in \Z\setminus \{0\}$. Plainly, $p=2$, and we may assume $t>0$. Inequality~\eqref{e4_5} with $i=0,1$ implies
$$
0\le tr_0-tr_1+t <1 \; \mbox{and}\; 0\le -tr_0+tr_1-t <1.
$$
Thus, $r_0-r_1+1=0$, as stated.
\end{proof}

\begin{lemma}
$ $
\begin{itemize}
\item If $-2<r_0\leq-1$ and $-1< r_1\leq0$, $(0,-1)$ is a cycle of $\tau_{(r_1,r_2)}$.
\item If $-1\le r_0< 0$ and $0\leq r_1<1 $, $(0,1)$ is a cycle of $\tau_{(r_1,r_2)}$.
\end{itemize}
Thus, none of these parameters $(r_0,r_1)$ are in $\mathcal{D}_2^{(*)}$.
\end{lemma}

\begin{proof}
Simple computation.
\end{proof}

\subsection{Subregions of $\mathcal{D}_2^{(*)}$: direct approaches}

\begin{theorem} \label{Tsimpleorbit}
Assume that $\mathbf{r}=(r_0,r_1)$ is contained in one of the following sets.
\begin{itemize}
\item[(i)] $ \{  (r_0,r_1) \in \mathbb{R}^2 \mid  r_0>0, \, r_1\leq-2\sqrt{r_0}, \, r_0+r_1\geq0   \}. $
\item[(ii)] $ \{  (r_0,r_1) \in \mathbb{R}^2 \mid  r_0+r_1\leq-2, \, r_0-r_1\neq-1 \} \setminus \{  (r_0,r_1)\in \mathbb{R}^2 \mid   r_0 > -2, \, r_1 > -1  \}.$
\item[(iii)] $ \{  (r_0,r_1) \in \mathbb{R}^2 \mid   r_0+r_1\geq0, \, r_0<0, r_1\geq1  \} .$
\end{itemize}
Then, $\mathbf{r}\in\mathcal{D}_2^{(*)}$.
\end{theorem}

\begin{proof}
Throughout the entire proof, let us assume w.l.o.g. that $\abs{\alpha_1}\leq\abs{\alpha_2}$. Notice that in all cases $\alpha_1$ and $\alpha_2$ are real.

We will prove the theorem by contradiction. Thus, we assume that there is a non-trivial cycle $(a_0,\ldots,a_{p-1})$ of $\tau_\vector{r}$.

{Proof of (i):} The assumptions on $r_0$ and $r_1$ in (i) imply that $1<\alpha_1\leq\alpha_2$ (see Figure~\ref{fig:overview}). Inequality \eqref{eqProp262} implies that $a_{n+1}\le \alpha_1a_n$ for all $n\in\Nz$. If $a_j\leq0$ for some $j\in\set{0,\ldots,p-1}$, then, $(a_n)$ is a strictly decreasing sequence of negative numbers, which is impossible. Thus, $a_j>0$ for all $j\in\set{0,\ldots,p-1}$. However, in this case we have $r_0+r_1<0$ by Lemma~\ref{LSameSign}, which contradicts our assumption.

\medskip
{Proof of (ii):} If $r_0=0$, then $0\le r_1a_{n}+ a_{n+1}<1$ with $r_1\le -2$ and the result follows immediately. We divide up the remaining region into four subregions.

Case (iia): Assume $r_0+r_1\leq-2$ and $r_0>0$. In this case we have $0<\alpha_1<1<\alpha_2$. As in (i) we can conclude that $a_{n+1}\le \alpha_1a_n$ for all $n\in\Nz$. If $a_0\geq0$, then $(a_n)$ is a strictly decreasing sequence, which is impossible. Hence, $a_0<0$. Then all $a_n$ are negative, however, and we can apply Lemma~\ref{LSameSign} again to get a contradiction to our assumption $r_0+r_1\leq-2$.

Case (iib): Assume $r_0+r_1\leq-2$, and $r_0<0$, $r_1\le -1$, and $r_0-r_1+1>0$. Here we have $-1<\alpha_1<0$ and $\alpha_2>1$. By Lemma~\ref{LSameSign} the period $(a_0,\ldots,a_{p-1})$ must have both negative and non-negative members. By $\alpha_2>1$, we have $a_{n+1}\le \alpha_1a_n$ for all $n\in\Nz$, as in (i). Thus, if $a_n\geq0$, then, $a_{n+1}<0$.

Now we exclude the possibility of $a_n=0$ for some $n$. Supposing the contrary we may assume w.l.o.g. that $a_0=0$. Then, $a_1\le \alpha_1a_0=0$, but $a_1=0$ is excluded, because otherwise $(a_n)$ would be the zero sequence. Thus, $a_1<0$ and we get $a_2<-r_1a_1+1\le a_1+1\leq0$ using \eqref{Eqrek}. Repeated application of \eqref{Eqrek} shows that all members of the period are negative, which is impossible.

Thus, $a_n\neq0$ for all $n\in\Nz$. This implies that consecutive members of the period have different signs: if we assume to the contrary that there are two consecutive members which have the same sign (which must be $-1$, since we have already shown $a_n\geq0\Rightarrow a_{n+1}<0$), say $a_n,a_{n+1}<0$, then $a_{n+2}\leq a_{n+1}$ by \eqref{Eqrek} and $a_{n+2}\neq a_{n+1}$ since the period must have non-negative members, i.e., $a_{n+2}<a_{n+1}$, which is a contradiction. Hence, $p$ is even and we may assume w.l.o.g. that $0<a_0\leq a_{2l}$ for all $l\in\set{0,\ldots,\frac{p}{2}-1}$. As a result, we get
\begin{align}
\label{ESeqE}
a_0r_0+a_1r_1+a_2=a_0(r_0-r_1+1)+(a_0+a_1)r_1+(a_2-a_0),
\end{align}
and since $r_0-r_1+1>0$ and $a_2\geq a_0$, the first and the third summands are non-negative. Furthermore, we have $a_2\le \alpha_1a_1<-a_1$, which implies $a_0+a_1\le a_2+a_1<0$. Altogether, we get
\begin{align}
a_0r_0+a_1r_1+a_2\geq-(a_0+a_1)\geq1,
\end{align}
which is a contradiction.

Case (iic): Assume that $r_0\leq-2$, $r_1\le 0$ and $r_0-r_1<-1$. This implies $\alpha_1<-1$ and $\alpha_2>1$. As before, the period $(a_0,\ldots,a_{p-1})$ has to have both negative and non-negative members by Lemma~\ref{LSameSign}. Furthermore, $a_{n+1}\le\alpha_1a_n$ for all $n\in\Nz$. Thus, the largest element of the period (in absolute value) must be negative. We may assume w.l.o.g. that this element is $a_0$ and, hence, $a_0\le a_j<-a_0$ for all $j\not=0$. Then $a_0+a_1<0$ and, thus, all summands of \eqref{ESeqE} are non-negative.

If $a_2\neq a_0$ or $r_1\leq-1$, we get $a_0r_0+a_1r_1+a_2\geq1$, which is a contradiction. Thus, $a_0=a_2$ and $-1<r_1\le 0$. Hence,
\begin{align}
a_0r_0+a_1r_1+a_2=a_0(r_0+1)+a_1r_1\geq-a_0+a_1r_1\ge -a_0-\abs{a_1},
\end{align}
where we first used $r_0\leq-2$, and then $|r_1|<1$. The RHS of the last inequality is at least one which is absurd, except when  $-a_0-\abs{a_1}=0$ or $-a_0=|a_1|$. The case $a_1>0$ is impossible because $a_0=a_2\le \alpha_1a_1<-a_1$. Hence $a_1<0$, but then $a_1=a_0=a_2$ and the cycle is constant, contradicting Lemma~\ref{LSameSign}.

Case (iid): Assume that $r_0+r_1\leq-2$ and $r_1>0$. Then $\alpha_1>1$, $\alpha_2<-1$, and $\alpha_1<-\alpha_2$. By interchanging $\alpha_1$ and $\alpha_2$ we obtain $a_{n+1} \le \alpha_2 a_n$, as in (i).

By setting $A\ce\max \{|a_j|, j=0,\dots,p-1\}$ we can show as in case (iic) that if $|a_j|=A$, then $a_j<0$. We may assume w.l.o.g. that $|a_0|=A$. As a result
$$
e_0 = a_0 r_0 + a_1 r_1 + a_2 = a_0(r_0+r_1+1) + r_1(a_1-a_0) + a_2-a_0.
$$
At the same time, $r_0+r_1+1\le -1$. Thus, the first summand is at least $-a_0\ge 1$. As $|a_0|=-a_0\ge |a_j|$ for all $j\ge 1$, we have $a_1-a_0, a_2-a_0\ge 0$. Finally, since $r_1>0$, we conclude that $e_0\ge 1$, which is a contradiction.

\medskip
{Proof of (iii):} Assume $r_0+r_1\ge 0$ and $r_0< 0$, $r_1>1$. Then $0<\alpha_1<1, \alpha_2< -1$. By interchanging $\alpha_1$ and $\alpha_2$, inequality \eqref{eqProp261} implies $a_{n+1} \ge\alpha_2 a_{n}$. By Lemma~\ref{LSameSign} we know that the sequence $(a_n)$ has both non-positive and positive members. Thus, if $a_{n-1}\le 0$, then $a_n > |a_{n-1}|\ge 0$. It also follows that
$$
a_{n+1}< -r_0 a_{n-1} -r_1a_n + 1 < -a_n +(1-r_0 a_{n-1}).
$$
As $1-r_0 a_{n-1}< 1$ and $a_n,a_{n+1} \in \mathbb{Z}$, we get $a_{n+1}\le -a_n$. Hence, $(|a_n|)$ is monotonically increasing and it has a jump when $a_n\le 0$. This is impossible with a periodic sequence.
\end{proof}

The proof of the next result, which characterizes a further region that is free from non-trivial cycles, is divided into several lemmas and will constitute the remaining part of the present section.

\begin{theorem} \label{Tnotsosimpleorbit}
$\{(r_0,r_1) \in \mathbb{R}^2\mid r_0-r_1 < -1,\, r_0\geq0 \} \subset \mathcal{D}_2^{(*)}$.
\end{theorem}

In the remaining part of this section, we assume that $r_0-r_1 < -1$ and $r_0\geq 0$. It follows that $-1<\alpha_1\leq0, \alpha_2< -1$. We define
\begin{eqnarray*}
  S_0 &\ce& \{(0,0)\}, \\
  S_1 &\ce& \{(a_1,a_2)\mid a_1\ge 0, a_2\le -a_1\}\setminus\{(0,0)\} \\
  S_2 &\ce& \{(a_1,a_2)\mid a_1\le 0, a_2\ge -a_1\}\setminus\{(0,0)\} \\
  S_3 &\ce& \{(a_1,a_2)\mid a_1> 0, a_2\ge 0\} \\
  S_4 &\ce& \{(a_1,a_2)\mid a_1< 0, a_2\le 0\} \\
  S_5 &\ce& \{(a_1,a_2)\mid a_1\ge 0, -a_1<a_2< 0\} \\
  S_6 &\ce& \{(a_1,a_2)\mid a_1< 0, 0<a_2<-a_1\}.
\end{eqnarray*}
Clearly, $S_0,\dots,S_6$ form a partition of $\mathbb{Z}^2$.

\begin{lemma} \label{J5}
  We have $\tau_{\bf r} (S_1) \subset S_2$ and $\tau_{\bf r} (S_2) \subset S_1$.
\end{lemma}

\begin{proof}
  Let $(a_1,a_2) \in S_1$ and $\tau_{\bf r} (a_1,a_2) = (a_2,a_3)$. Then $a_3 = -\lfloor r_0a_1 + r_1a_2 \rfloor$. From the conditions on $(r_0,r_1)$ and $(a_1,a_2)$ we have
  \begin{eqnarray*}
    a_3 &\ge& -r_0a_1 -r_1a_2 \\
     &>& -r_0a_1-(r_0+1)a_2 ;\  (\mbox{since}\; r_0+1<r_1 \; \mbox{and} \; a_2<0) \\
     &\ge& a_2r_0-(r_0+1)a_2 ;\  (\mbox{since}\; a_2\le -a_1)\\
     &=& -a_2.
  \end{eqnarray*}
  We have proved $a_3>-a_2$, which together with $a_2<0$ implies $(a_2,a_3)\in S_2$, hence, the first inclusion is proved.

  In order to prove the second inclusion, let $a_1,a_2,a_3$ such that $(a_1,a_2) \in S_2$, and $a_3 = -\lfloor r_0a_1 + r_1a_2 \rfloor$. As a consequence, we get
  \begin{eqnarray*}
    a_3 &<& -r_0a_1 -r_1a_2 +1\\
     &\le& -r_0a_1-(r_0+1)a_2 +1;\  (\mbox{since}\; r_0+1<r_1 \; \mbox{and} \; a_2\ge 0) \\
     &\le& a_2r_0-(r_0+1)a_2 +1;\  (\mbox{since}\; a_2\ge -a_1)\\
     &=& -a_2+1.
  \end{eqnarray*}
  Thus, $a_3\le -a_2$, which together with $a_2>0$ implies $(a_2,a_3) \in S_1$.

\end{proof}

\begin{lemma} \label{J6}
  We have $\tau_{\bf r} (S_3) \subset S_1 \cup S_0$ and $\tau_{\bf r} (S_4) \subset S_2\cup S_0 $.
\end{lemma}

\begin{proof}
Let $(a_1,a_2) \in S_3$, i.e., $a_1>0, a_2\ge 0$. Then $a_3 = -\lfloor r_0a_1 + r_1a_2 \rfloor< -a_2$, which means $(a_2,a_3) = \tau_{\bf r} (a_1,a_2) \in S_1$, and the first assertion is proved.

Now let $(a_1,a_2)\in S_4$, which implies $a_1<0,a_2\le 0$. Set $a_3 = -\lfloor r_0a_1 + r_1a_2 \rfloor$. If $a_2<0$, then $a_3>-a_2$, i.e., $(a_2,a_3)\in S_2$. If, however, $a_2=0$, then $a_3\ge 0$, but $a_3=0$ only if $r_0=0$, which is excluded. Thus, $(a_2,a_3) \in S_2$.
\end{proof}

\begin{lemma} \label{J7}
$\tau_{\bf r} (S_5) \subset S_2\cup S_4\cup S_6$ and $\tau_{\bf r} (S_6) \subset S_1 \cup S_3 \cup S_5$.
\end{lemma}

\begin{proof}
Let $(a_1,a_2) \in S_5$ and $\tau_{\bf r} (a_1,a_2) = (a_2,a_3)$. Since $a_2 < 0$, the pair $(a_2,a_3)$ belongs to $S_2\cup S_4\cup S_6$.

If $(a_1,a_2)\in S_6$, then $a_2>0$ and $\tau_{\bf r} (a_1,a_2)$ belongs to $S_1 \cup S_3 \cup S_5$.
\end{proof}

Now we are in the position to prove Theorem \ref{Tnotsosimpleorbit}. Lemmas \ref{J5}, \ref{J6} and \ref{J7} show that the orbits of $\tau_{\bf r}$ are all governed by the following graph.

\begin{figure}[H]
\centering
\includegraphics[width=\textwidth]{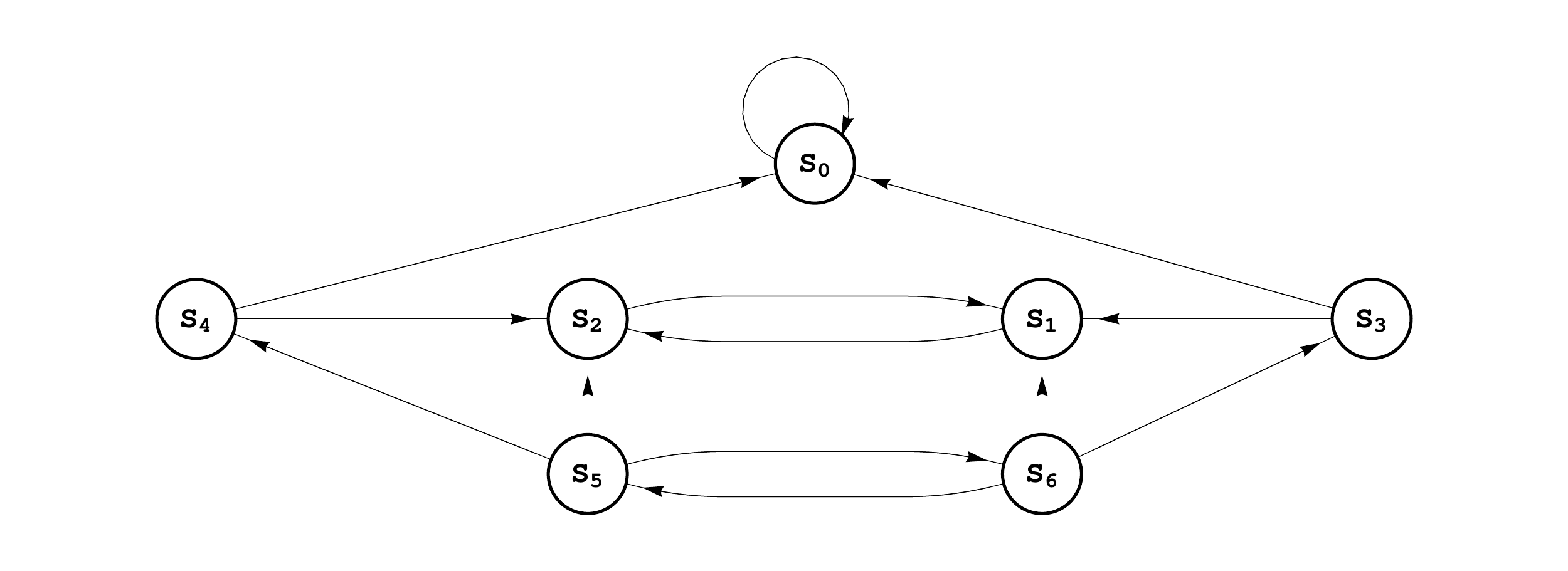}
\caption{Illustration of the action of $\tau_{\mathbf{r}}$ for the proof of Theorem~\ref{Tnotsosimpleorbit}.}
\end{figure}

Each orbit has to end up in one of the following cycles of the graph:
\begin{itemize}
  \item[(a)] in $S_0 \rightarrow S_0$,
  \item[(b)] in $S_1 \rightarrow S_2 \rightarrow S_1$,
  \item[(c)] or in $S_5 \rightarrow S_6 \rightarrow S_5$.
\end{itemize}
In case $(a)$ we have already proven the desired result.

In case $(b)$, as soon as we reach the cycle $S_1 \rightarrow S_2 \rightarrow S_1$, we have
$$
0< a_1 \le -a_2 < a_3 \le -a_4 < \dots
$$
according to the proof of Lemma \ref{J5}. Thus, $|a_k| \rightarrow \infty$, and $\tau_{\bf r}$ has no cycle for this orbit.

In case $(c)$ we have
$$
|a_k| > |a_{k+1}| > |a_{k+2}| \dots
$$
However, since $a_k$ is finite, this sequence must stop. Therefore, no orbit can end up in this cycle. \hfill $\Box$

$ $\\
So far, we have only treated a small part of the quadrant $\{(r_0,r_1) \in \mathbb{R}^2 \,:\, -r_0 < r_1< r_0+1\}$. The points strictly inside the triangle with vertices $(-1,0), (1,-2), (1,2)$ define contractive mappings. Thus, these points are part of the classification problem: which of them belong to the set $\mathcal{D}_2^{(0)}$? The points between the parallel lines $r_1=-r_0-1$ and $r_1=-r_0$ have the finite orbit $(1)$. Finally, by Theorem \ref{Tsimpleorbit}~(i), the mappings corresponding to points of the region $r_0>0$, $r_1\leq-2\sqrt{r_0}$, and $r_0+r_1\geq0$ belong to $\mathcal{D}_2^{(*)}$. In the following, we deal with the remaining part of this quadrant up to some finite region.

Approaching the critical line segment $r_0=1, -2\le r_1\le 2$, one can find points ${\bf r}=(r_0,r_1)$ such that $\tau_{\bf r}$ is expanding, but has arbitrarily long cycles. Indeed, Akiyama and Peth\H{o} \cite{Akiyama-Pethoe-:12} proved that the mapping $\tau_{\bf r}$ has infinitely many cycles for ${\bf r}=(1,r_1)$ with arbitrary $-2 < r_1 <2$. Let $r_1$ be irrational, and $(a_0,\dots,a_{p-1})$ be a cycle of $\tau_{\bf r}$. Consequently, there exists a $0<\delta$ such that $\delta< a_{k-1}+r_1a_k +a_{k+1}< 1-\delta$ holds for all $k=1,\dots,p$. Choosing a small enough $\varepsilon>0$, we get $0\le (1+\varepsilon) a_{k-1}+r_1a_k +a_{k+1}< 1$, i.e., $(a_0,\dots,a_{p-1})$ is a non-trivial periodic orbit of $\tau_{(1+\varepsilon,r_1)}$ as well.

Weitzer \cite{Weitzer2015a} defined six infinite sequences of polygons which cover the critical line. Each SRS associated to points in these polygons has cycles. Moreover, most of these polygons have points not only on and to the left, but also to the right of the critical line.
\begin{figure}[H]
\centering
\begin{overpic}[height=7.7cm]{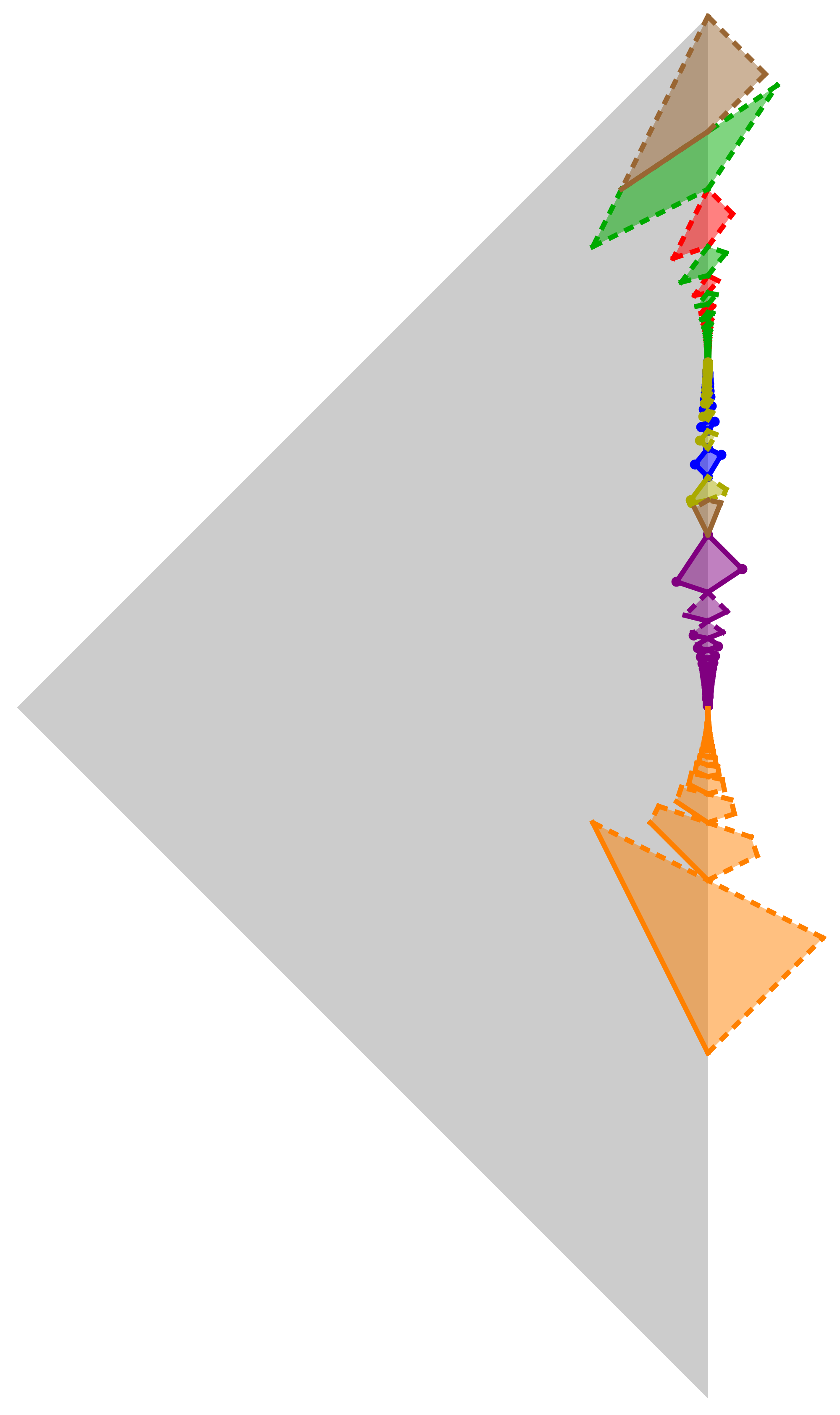}
\put(-14.00,48.60){$(-1,0)$}
\put(51.50,0.00){$(1,-2)$}
\put(51.50,24.00){$(1,-1)$}
\put(51.50,48.60){$(1,0)$}
\put(51.50,74.20){$(1,1)$}
\put(51.50,99.00){$(1,2)$}
\end{overpic}
\caption{Six families of cutout polygons covering the critical line $r_0=1$ almost everywhere for $-1\leq r_1\leq 2$.}
\end{figure}
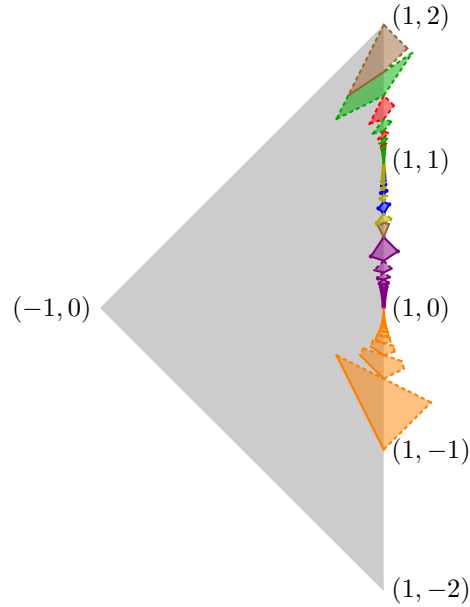

In the following, we prove that points with this property belong to a bounded region.

\begin{theorem}
\[
\Big\{
(r_0,r_1) \in \mathbb{R}^2\mid r_0-r_1 > -\frac12,\,  r_1\ge\max\{-2\sqrt{r_0},-r_0\},  \,   r_0>\frac{3}{2} + \sqrt{2}
\Big\}
\subset \mathcal{D}_2^{(*)}.
\]

\end{theorem}

\begin{proof}
Let $\mathbf{r}=(r_0,r_1)$ be an element of the set specified in the statement of the theorem.
We first deal with the case where the polynomial $P(x)=x^2+r_1x +r_0$ has two real roots $\alpha_1$ and $\alpha_2$. Then, $\alpha_2\le\alpha_1<-1$. Assume that $\tau_{\mathbf{r}}$ admits the cycle $(a_n)$. We have
$$
|a_n| \le \frac{1}{||\alpha_1|-1| ||\alpha_2|-1|} = \frac{1}{(\alpha_1+1)(\alpha_2+1)} = \frac{1}{r_0-r_1+1}< 2
$$
by Corollary~\ref{Cor23}. Thus, $\tau_{\mathbf{r}}$ only admits cycles consisting of elements taken from the set $\{-1,0,1\}$. A simple computation shows that this is impossible.

Now we proceed with the case where $P(x)$ has a pair of complex conjugate roots, i.e., $\alpha_2 = \bar{\alpha}_1$. Then, $|\alpha_1|=|\alpha_2|=\sqrt{r_0}$ and, using Corollary~\ref{Cor23} again, we obtain
$$
|a_n| \le \frac{1}{||\alpha_1|-1| ||\alpha_2|-1|} = \frac{1}{r_0-2 \sqrt{r_0}+1} = \frac{1}{(\sqrt{r_0}-1)^2}.
$$
Since $r_0> \frac32 + \sqrt{2}$ is equivalent to $\sqrt{r_0}-1 >  \frac{1}{\sqrt{2}}$, $|a_n|<2$, which, similarly to the case we considered above, also leads to a contradiction.
(Note that the line $r_1= r_0+\frac12$ intersects the parabola $r_1^2 = 4r_0$ in the points $(\frac32 \pm \sqrt{2},2 \pm \sqrt{2})$.)
\end{proof}

With more effort one can improve the last theorem, but a complete characterization of parameters without non-trivial periodic points is, in spite of the results of Weitzer \cite{Weitzer2015a}, and of Akiyama and Peth\H{o} \cite{Akiyama-Pethoe-:12} mentioned above, a very hard problem.

\medskip
It remains one more infinite region: the points enclosed between the lines $r_0-r_1=-1$ and $r_0-r_1=-\frac12$ over the parabola $r_1^2=4r_0$. As a consequence of our last theorem in this section, only a bounded part of it may have points with associated SRS having non-trivial cycles.

\begin{theorem}
\[
\Big\{
(r_0,r_1) \in \mathbb{R}^2\mid -1<r_0-r_1\le -\frac12,\, r_1\ge 1+\sqrt{5},\, r_1^2\ge 4r_0
\Big\}
\subset \mathcal{D}_2^{(*)}.
\]
\end{theorem}

\begin{proof}
Let $\mathbf{r}=(r_0,r_1)$ be an element of the set specified above and assume that $(a_n)$ is a non-trivial periodic sequence.
Given the assumptions, it follows that the roots $\alpha_1,\alpha_2$ of $x^2+r_1x +r_0$ are real and satisfy $\alpha_2\le\alpha_1<-1$. Since $\alpha_1+\alpha_2 =-r_1\le -(1+\sqrt{5})$, we have $\alpha_2\le -\frac{1+\sqrt{5}}{2}$. Thus, we obtain
$$
\frac{-1}{\alpha_2^2-1} <a_{n+1}-\alpha_1 a_n < \frac{-\alpha_2}{\alpha_2^2-1}
$$
from \eqref{eqProp264}.

The function $\frac{-x}{x^2-1}$ is monotonically increasing in $(-\infty,-\frac{1+\sqrt{5}}{2}]$ and its maximum $1$ is located at $x=-\frac{1+\sqrt{5}}{2}$. This implies that  $\frac{-1}{\alpha_2^2-1}\ge \frac{1}{\alpha_2} \ge -\frac{2}{1+\sqrt{5}}= \frac{1-\sqrt{5}}{2}$, which leads to
$$
-1 < a_{n+1}-\alpha_1 a_n < 1.
$$

This inequality implies $a_n\not= 0$ for all $n$, and consecutive members of the sequence $(a_n)$ must have different signs. Moreover, we can write it in the form
$$
(\alpha_1+1)a_n -1 < a_{n+1}+a_n < (\alpha_1+1)a_n +1.
$$

 If $a_n<0$, then $(\alpha_1+1) a_n>0$ and $a_{n+1}+a_n\ge 0$, i.e., $a_{n+1}\ge -a_n$.

 If $a_n>0$, then $(\alpha_1+1) a_n<0$ and $a_{n+1}+a_n\le 0$, i.e., $a_{n+1}\le -a_n$. Hence, the sequence $(|a_n|)$ is monotonically increasing and it can be periodic only if it is constant, i.e., if $(a_n) = (a,-a,a,-a,\ldots)$  for some $a\in \Z$. However, this is excluded as a consequence of Lemma \ref{LAltSign}, which leads to the desired contradiction.
\end{proof}

\subsection{Subregions of $\mathcal{D}_2^{(*)}$: algorithmic approaches}
$ $\\
In the previous sections we have characterized $\Dstar{2}$ up to a bounded region which we denote by $\mathcal{R}\subseteq\R^2$. $\mathcal{R}\cap\operatorname{int}(\D{2})$ has been characterized in large parts in \cite{Weitzer2015a} by two algorithms which can be adapted for those parts of the exterior of $\D{2}$, for which the corresponding SRS is expanding. This leads to the following result:

\begin{theorem}
\label{TAlgorithm}
If $\mathcal{R}\subseteq\R^2$ denotes the region not covered by any of the theorems above, then $\mathcal{R}\cap\set{(x,y)\in\R^2\mid x\geq\frac{4}{3}}$ is contained in $\Dstar{2}$.
\end{theorem}

\begin{proof}
We outline the idea behind the adapted version of one of the algorithms in \cite{Weitzer2015a}. We need the algorithms to work for parameters of SRS which are expanding instead of contracting. The first ingredient is a result by Lagarias and Wang \cite{LagariasWang:1996} which states that for any $\vector{r}\in\R^2$, for which
\[
R(\mathbf{r})=
\begin{pmatrix}
0 & 1\\
-r_0 & -r_1
\end{pmatrix}
\]
is expanding, and any $1<\rho<\min\set{\abs{\lambda}\mid\lambda\text{ eigenvalue of } R(\vector{r})}$, there is a norm $\norm{\cdot}_{\vector{r},\rho}$ on $\R^2$ such that $\norm{R(\vector{r})\vector{x}}_{\vector{r},\rho}>\rho\norm{\vector{x}}_{\vector{r},\rho}$ for all $\vector{x}\in\Z^2$. If $\norm{\vector{x}}_{\vector{r},\rho}>\frac{\norm{(0,\ldots,0,1)}_{\vector{r},\rho}}{\rho-1}$, we thus get
\begin{align*}
\norm{\tau_\vector{r}(\vector{x})}_{\vector{r},\rho}\geq\norm{R(\vector{r})\vector{x}}_{\vector{r},\rho}-\norm{(0,\ldots,0,1)}_{\vector{r},\rho}>\norm{\vector{x}}_{\vector{r},\rho}.
\end{align*}
Hence, we can restrict our search for possible cycles of $\tau_\vector{r}$ to the finite set of witnesses
\begin{align*}
W_{\vector{r},\rho}\ce\set{\vector{x}\in\Z^2\mid\norm{\vector{x}}_{\vector{r},\rho}\leq\frac{\norm{(0,\ldots,0,1)}_{\vector{r},\rho}}{\rho-1}},
\end{align*}
which is the basis of the single parameter version of the algorithm.

To settle entire convex regions of $\R^2$, we observe that the norm $\norm{\cdot}_{\vector{r},\rho}$ depends continuously on $\vector{r}$ and, thus,
\begin{align*}
\norm{\tau_\vector{s}(\vector{x})}_{\vector{r},\rho}>\norm{\vector{x}}_{\vector{r},\rho}
\end{align*}
also holds for all $\vector{x}\in\Z^2$ and all $\vector{s}\in\R^2$ sufficiently close to $\vector{r}$. Thus, there is a bounded set $K\subseteq\R^2$ of such $\vector{s}$ which contains $\vector{r}$ as an interior point, and which has a positive distance from the boundary of $\mathcal{D}_2$ (i.e., the SRS of all parameters in $K$ are strictly expanding). We consider the following equivalence relation on $K$:
\begin{align*}
\vector{s}\sim\vector{t}\Leftrightarrow\fa\vector{x}\in W_{\vector{r},\rho}:\tau_\vector{s}(\vector{x})=\tau_\vector{t}(\vector{x}).
\end{align*}
Since $K$ is bounded and has a positive distance form $\mathcal{D}_2$, it follows that $K\slash_\sim$ is a finite set and every element of $K\slash_\sim$ is either contained in $\Dstar{2}$ or has an empty intersection with it by construction. Each of the finitely many parts (whose number, however, tends to infinity upon decreasing the distance to $\mathcal{D}_2$) can, thus, be settled by the single parameter version of the algorithm by taking an arbitrary parameter in the respective part.
\end{proof}

\section*{Acknowledgment}
We are very indebted to the anonymous referee for the careful reading of our manuscript and for the suggestions, which improved the quality of the presentation.

\bibliographystyle{siam}

\bibliography{D2Star}
\end{document}